\documentclass[a4paper,10pt,leqno]{amsart}
\usepackage[ansinew]{inputenc}
\usepackage{blowup}
\blowUp{paper={x1.10}} 

\usepackage{url}

\usepackage[colorlinks=true,linkcolor=blue,citecolor=red]{hyperref}

\hypersetup{pdfinfo={
  Title=On an Inequality Related to a Certain Fourier Cosine Series,
  Author=Wolfgang Gabcke,
  Subject=Proof of an inequality related to a certain Fourier cosine series,
  Keywords={inequality, Fourier cosine series, Chebyshev Polynomials, generating function}}
}

\DeclareMathOperator{\coloneqq}{%
   \mathrel{\mathop:}=\,}
\newcommand{\scfrac}[2]{%
   {\textstyle\frac{\displaystyle\raisebox{0.210ex}{\(#1\)}}{\displaystyle\raisebox{-0.710ex}{\(#2\)}}}%
}

\theoremstyle{plain}
\newtheorem{thm}{Theorem}[section]
\newtheorem{cor}[thm]{Corollary}
\newtheorem{lem}[thm]{Lemma}

\theoremstyle{definition}

\theoremstyle{remark}

\begin{document}
%
\title[ON AN INEQUALITY RELATED TO A CERTAIN FOURIER COSINE SERIES]
      {ON AN INEQUALITY RELATED TO\\ A CERTAIN FOURIER COSINE SERIES}
\author[WOLFGANG GABCKE]{WOLFGANG GABCKE}
\address{Wolfgang Gabcke, Göttingen, Germany}
\email{wolfgang@gabcke.de}

\date{October 17, 2015}

\begin{abstract}
We prove that the Fourier cosine series
   \begin{align*}
      \sum_{k=1}^{\infty}(-1)^{k+1}\scfrac{r^k\cos{k\phi}}{k+2}
   \end{align*}
assumes its maximum value at \(\phi = 0\) for \(\phi \in [0, \pi)\) regardless of \(r\) if \(r \in (0, 1]\). This
was first proved by Arias de Reyna and van de Lune. The more compact proof presented here is based on a
ge\-ne\-ra\-ting function of the Chebyshev Polynomials.
\end{abstract}
%
\maketitle
%
%
%
%
%
%
%
\section{Reformulation of the Inequality}
\noindent
Arias de Reyna and van de Lune \cite{Rey} estabished and proved the inequality
\begin{thm}\label{t0001}
If \(r \in (0, 1]\) and \(\phi \in (0, \pi)\) then
   \begin{align}
   \label{m0009}
      \sum_{k=1}^{\infty}(-1)^{k+1}\scfrac{r^k\cos{k\phi}}{k+2}<
      \sum_{k=1}^{\infty}(-1)^{k+1}\scfrac{r^k}{k+2}.
   \end{align}
\end{thm}
To get a simpler proof of this theorem it is convenient to eliminate the cosine
function by setting \(x \coloneqq \cos{\phi}\). Then \(-1 < x < 1\) and
   \begin{align}
   \label{m0001}
      \cos{k\phi} = T_k{(\cos{\phi})} = T_k{(x)},
   \end{align}
where \(T_k{(x)}\) is the \(k\)-th Chebyshev Polynomial of the first kind. By this substitution we obtain
the function of two variables
   \begin{gather}
   \label{m0002}
      \quad f(x, r) \coloneqq \sum_{k=1}^{\infty}(-1)^{k+1}\scfrac{r^k\,T_k{(x)}}{k+2}\quad
                                              \rlap{\(\:\bigl(r \in (0, 1],\:x \in (-1, 1]\bigr)\)}
   \end{gather}
from the left side of \eqref{m0009}, where we have included the value \(x = 1\) additionally. This function satisfies
\begin{lem}\label{t0002}
\(f(x, r)\) is a monotonically increasing function of \(x\). So the inequality
   \begin{align}
   \label{m0003}
      f(x, r) < f(1, r) = \sum_{k=1}^{\infty}(-1)^{k+1}\scfrac{r^k}{k+2}
                           \raisebox{1.5ex}{\footnotemark}
                        = \left(\log{(1+r)} - r + \scfrac{r^2}{2}\right)\! /\, r^2
   \end{align}
holds for \(r \in (0, 1]\) and \(x \in (-1, 1)\).
\end{lem}
\footnotetext{ We have \(T_k{(1)} = 1\) for all \(k\).}
\noindent
From this, theorem \ref{t0001} above follows immediately.
%
%
%
%
%
%
\section{An Integral Representation of the Series}
\noindent
As we can see in \cite{Rey}, trying to prove lemma \ref{t0002} from definition \eqref{m0002} directly must be hard. So we are looking for a different representation of \(f(x, r)\).
Indeed, there is such a representation as a definite integral.
\begin{lem}\label{t0003}
   \begin{align}
   \label{m0004}
      f(x, r) = \scfrac{1}{r^2}\!\int\limits_{\!\!\!\!0}^{\;\,\,r}\!t^2 \scfrac{t+x}{t^2+2xt+1}\,dt.
   \end{align}
\end{lem}
\begin{proof}
The series in \eqref{m0002} looks similar to one of the generating functions of the
Chebyshev Polynomials of the first kind. We take equation 22.9.9 from \cite{Abr}\footnote{ This can be
shown by setting \(z = r\,e^{i\phi}\) (\(0 \le r < 1\), \(0 \le \phi \le \pi\)) in the geometric
series \(1/(1-z)=\sum_{k=0}^{\infty}z^k\), taking the real part and using \eqref{m0001}.}
\begin{align*}
      \scfrac{1-xz}{1-2xz+z^2} = \sum_{k=0}^{\infty}\,T_k{(x)}\,z^k\rlap{\(\qquad\qquad\,\,
      (|z| < 1, |x| \le 1)\).}
\end{align*}
After moving the constant term \(T_0(x) \equiv 1\) of the series to the left side and setting \(z = -r\) we obtain
\begin{align*}
      r\,\scfrac{r+x}{r^2+2xr+1} = \sum_{k=1}^{\infty}(-1)^{k+1}\,T_k{(x)}\,r^k.
\end{align*}
Multiplying this equation by \(r\), writing \(t\) for \(r\) and integrating from \(0\) to \(r\) over
\(t\) yields
\begin{align*}
      \int\limits_{\!\!\!\!0}^{\;\,\,r}\!t^2 \scfrac{t+x}{t^2+2xt+1}\,dt = \sum_{k=1}^{\infty}(-1)^{k+1}
      \scfrac{r^{k+2}}{k+2}\,T_k{(x)} = r^2f(x, r)
\end{align*}
as required. Obviously this equation remains valid for \(r = 1\) if \(x \in (-1, 1]\).
\end{proof}
So we get the
\begin{proof}[Proof of Lemma \ref{t0002}]
Differentiation of \eqref{m0004} with respect to \(x\) shows that
\begin{align*}
      \scfrac{\partial}{\partial x}f(x, r) = \scfrac{1}{r^2}\!\int\limits_{\!\!\!\!0}^{\;\,\,r}\!t^2
      \scfrac{1-t^2}{[t^2+2xt+1]^2}\,dt > 0
\end{align*}
since the integrand is positive for \(t \in (0, r)\). Therefore \(f(x, r)\) is a monotonically increasing
function of \(x\) and in particular inequality \eqref{m0003} must hold. The explicit expression of \(f(1, r)\)
follows from the logarithm series.
\end{proof}
\noindent
Thus theorem \ref{t0001} is proved.
\begin{cor}
The integral in \eqref{m0004} can be solved, giving
   \begin{align*}
      f(x, r) = \scfrac{1}{r^2}\Bigl[\scfrac{r^2}{2}-xr+\Bigl(x^2-\scfrac{1}{2}\Bigr)
                \log{\!(r^2+2xr+1)}
                +2xw\arctan{\!\Bigl(\scfrac{wr}{1+xr}\Bigr)}\Bigr]
   \end{align*}
with the abbreviation \(w \coloneqq \sqrt{1-x^2}\).
\end{cor}
\noindent
But this representation does not help to prove lemma \ref{t0002} since its derivative with respect
to \(x\) takes on a very complicated shape.
%
%
%
%
%
%


\begin{thebibliography}{1}
\bibitem{Abr} \textsc{M.\(\,\,\)Abramowitz, I.\(\,\,\)A.\(\,\,\)Stegun},
\textit{Handbook of Mathematical Functions}, National Bureau of Standards, Applied Mathematics Series,
\textbf{55}, Tenth Printing, 1972,
\url{http://www.cs.bham.ac.uk/~aps/research/projects/as/book.php}
%
\medskip
\bibitem{Rey} \textsc{J.\(\,\,\)Arias\(\,\,\)de\(\,\,\)Reyna, J.\(\,\,\)van\(\,\,\)de\(\,\,\)Lune},
\textit{A proof of a trigonometric inequality. A glimpse inside the mathematical kit\-chen}, J. Math. Inequal.,
Vol. \textbf{5(3)}, 2011, 341--353,
\url{http://dx.doi.org/10.7153/jmi-05-30}

\end{thebibliography}
\end{document}